\documentclass[11pt,reqno,twoside]{amsart}

\usepackage{amsmath}
\usepackage{amsfonts}
\usepackage{amssymb}

\usepackage{epsfig}
\usepackage{tikz-cd}
\usepackage{tikz}

\usepackage{xcolor}

\numberwithin{equation}{section}

\newif\ifdraft\drafttrue

\draftfalse

\newcommand\hd{Hausdorff dimension}

\newcommand\eq[2]{{\ifdraft{\ \tt [#1]}\else\ignorespaces\fi}\begin{equation}\label{eq:#1}{#2}\end{equation}}
\newcommand {\equ}[1]     {\eqref{eq:#1}}

\newcommand{\R}{{\mathbb {R}}}
\newcommand{\Z}{{\mathbb {Z}}}
\newcommand{\N}{{\mathbb {N}}}

\newcommand{\Ad}{{\operatorname{Ad}}}
\newcommand{\ad}{{\operatorname{ad}}}

\newcommand{\Tr}{\operatorname{Tr}}
\newcommand{\SL}{\operatorname{SL}}

\newcommand{\ggm}{G/\Gamma}

\newcommand{\diam}{\operatorname{diam}}
\newcommand{\dist}{\operatorname{dist}}
\newcommand{\diag}{\operatorname{diag}}
\newcommand{\supp}{\operatorname{supp}}

\newcommand{\Lie}{\operatorname{Lie}}
\newcommand{\Leb}{\operatorname{Leb}}

\newcommand {\ignore}[1]  {}

\newcommand\hs{homogeneous space}

\newcommand\cic{C^\infty_{comp}}

\newcommand{\df}{{\, \stackrel{\mathrm{def}}{=}\, }}

\newcommand{\vs}{{\bf{j}}}

\newcommand{\vr}{{\bf i}}

\newcommand {\comm}[1]   {\textcolor{red}{#1}}
\newcommand {\gr}[1]     {\textcolor{green}{#1}}

\DeclareMathOperator{\codim}{codim}
\DeclareMathOperator{\spn}{span}

\newcommand{\vre}{\varepsilon}

\newtheorem{thm}{Theorem}[section]
\newtheorem{lem}[thm]{Lemma}
\newtheorem{prop}[thm]{Proposition}
\newtheorem{cor}[thm]{Corollary}

\newtheorem{defn}[thm]{Definition}

\newtheorem{claim}[thm]{Claim}

\newcommand{\mc}{\mathcal}

\renewcommand{\a}{\alpha}
\renewcommand{\b}{\beta}

\renewcommand{\l}{\lambda}

\newcommand{\set}[1]{\left\{#1\right\}}
\renewcommand{\r}{\rightarrow}

\def\multiset#1#2{\ensuremath{\left(\kern-.3em\left(\genfrac{}{}{0pt}{}{#1}{#2}\right)\kern-.3em\right)}}

\begin{document}

\title[Dimension estimates
for  escape on average]
{Dimension bounds for escape on average \\ in homogeneous spaces}
\author{Dmitry Kleinbock}
\address{Department of Mathematics, Brandeis University, Waltham MA}
\email{kleinboc@brandeis.edu}

\author{Shahriar Mirzadeh}
\address{Department of Mathematical Sciences, University of Cincinnati, Cincinnati OH}

 \email{mirzadsr@ucmail.uc.edu}
\date{September 2023}

\thanks{The first named author   was supported by NSF grant  DMS-2155111 and by a grant from the Simons Foundation (922989, Kleinbock).}

\begin{abstract} 
Let $X = \ggm$, where $G$ is a Lie group and $\Gamma$ is a  uniform lattice in $G$, and let $O$ be an open subset of $X$. We give  {an upper estimate  for the \hd\ of the} set of points whose trajectories escape $O$ on average with frequency $\delta$, where $0 < \delta \le 1$.
\end{abstract}



\maketitle
\section{Introduction}
Throughout the paper  we let  $G$ be a Lie group and $\Gamma $   a 
lattice in $G$,
denote by $X$ the homogeneous space  $G/\Gamma$ and by $\mu$   the
$G$-invariant probability measure on $X$. Let ${F^+:= }{({g_t})_{t
\ge 0}}$ be a one-parameter 
{sub}semigroup of $G$.  For a 
subset $O$ of $X$, define
 $$\widetilde E(F^+,O) := \big\{ x \in X: \exists \ t_0\text{ such that }g_tx \notin
 O \ \forall\, t\ge t_0\big\}
 $$
 to be the set of $x\in X$ {\sl eventually escaping} $O$ under the action of $F^+$. 
If the $F^+$-action on $(X,\mu)$ is ergodic, then it follows from Birkhoff's Ergodic Theorem that $\mu\big(\widetilde E(F^+,O)\big) = 0$ as long as $O$ has positive measure. Furthermore, under some additional assumptions 
{one can obtain}  upper estimates for the \hd\ of $\widetilde E(F^+,O)$. {More precisely, one 
{has} the following `Dimension Drop Conjecture', originated from a question asked by Mirzakhani {(private communication)}:
{\it if 
$F\subset G$ is a subsemigroup and $O$ is an open subset of $X$, then either $\widetilde E(F,O)$ has positive measure, or its \hd\ is less than the dimension of $X$.}
When $X$ is compact it follows from the variational principle for measure-theoretic entropy, as outlined in \cite[\S 7]{KW}; an effective argument using exponential mixing was developed in \cite{KMi1}. See 
{\cite{EKP} (resp.,   \cite{KMi2, KMi3}) for the proof of this conjecture for Lie groups of  rank one (resp., higher rank with some additional assumptions).}

{In this work we consider trajectories which are allowed to enter $O$, but not as frequently as mandated by Birkhoff's Ergodic Theorem. Namely, for $0 < \delta \le 1$  let us say that a point $x \in X$ \emph{$\delta$-escapes $O$ on average under the action of $F^+$} if it belongs to the set
$$E_{\delta}(F^+,O):=\left\{x \in X: \limsup_{T \rightarrow \infty} \frac{1}{T} \int_0^T 1_{O^c}(g_tx)dt \ge \delta\right\}. $$
Again under the assumption of  ergodicity 
for any 
measurable  $O\subset X$  
Birkhoff's Ergodic theorem  implies that
$$\lim\limits_{T\rightarrow\infty}\frac{1}{T}\int_0^T 1_{O^c}(g_tx)\,dt = \mu (O^c).$$
Hence, the set $E_\delta(F^+,O)$ has full measure for any $0< \delta \le \mu (O^c)$, and has measure  zero for any $\delta > \mu(O^c)$.
Clearly for any $\delta$ as above one has $$E_\delta(F^+,O)\supset  E_1(F^+,O) \supset  \widetilde E(F^+,O).$$ Thus one can ask the following questions: under some assumptions, can it be shown  that
      the set $E_1(F^+,O)$ has less than full \hd? is the same true for $E_\delta(F^+,O)$ for some $\delta < 1$?


\ignore{Here is a special case of our main result:
\begin{thm}
\label{Special Case} Let $G$ be a connected semisimple Lie group with finite center and with no compact factors, $\Gamma $   a uniform 
irreducible lattice in $G$, $X = G/\Gamma$, and let ${F^+}$ be a one-parameter {$\Ad$-}diagonalizable
{sub}semigroup of $G$. Then for any $O\subset X$ with non-empty interior there exists $\delta_O<1$ such that
 for any $\delta > \delta_O$, the \hd\ of $E_{\delta}(F^+,O)$ is strictly smaller than the dimension of $X$.
\end{thm}
In fact we present an effective way to estimate the \hd\ of $E_{\delta}(F^+,O)$ from above, as well as an explicit bound for $\delta_O$.}

\smallskip


{To state 
our results, we need to introduce some notation.
If $X$ is a metric space with metric `${\dist}$', $O$ is a subset of $X$ and $r > 0$, we will denote by $\sigma_r O$ the {\sl inner $r$-core\/} of $O$, defined as
$$
\sigma_r O :=  \{x\in X: {\dist}(x,O^c) > r\},$$
and by $O^{(r)}$ the  {\sl $r$-neighborhood\/}   of $O$, that is, 
 $$
O^{(r)} :=  \{x\in X: {\dist}(x,O) < r\}.
$$}
The notation ${A\gg B}$,
where  $A$ and $B$ are quantities depending on certain parameters, will mean ${A \ge  CB}$,  
where {$C$ is a constant} 
independent  
on those parameters. \hd\ (see Definition \ref{defhd}) will be denoted by `$\dim$', and $\codim A$ will stand for the Hausdorff codimension of a set $A$, that is, the difference between the dimension of the ambient space and $\dim A$.

Fix a {right-invariant} Riemannian {structure  on {a Lie group} $G$, and denote by
`${\dist}$' the corresponding Riemannian metric},
  {using} the same notation for the induced metric on homogeneous spaces of $G$. 
  {In what follows, if $P$ is a subgroup of $G$, we will denote by $B^{P}(r)$  the open ball of radius $r$ centered at the identity element with respect to the  
metric  on $P$ corresponding to the Riemannian structure induced from $G$,}
and by $\nu$  the Haar measure on $P$ induced by this metric.  $B(x,\rho)$ will stand for the open ball  in $X$ centered at $x\in X$ of radius $\rho$. 


Throughout this paper we shall assume that $\Gamma$ is a uniform lattice in $G$, that is, $X$ is compact. We note that the methods and results of this paper can be extended to the noncompact case, however 
this would require an extra effort controlling the escape of mass and will be pursued in subsequent work.

{For comparison, let us first state the   effective dimension estimate from  \cite{KMi1}. There the main tool was the} 
exponential mixing of the action. 
Namely, let us say that the flow $(X,F^+)$ is {\sl exponentially
mixing}  if there exist {$\gamma  > 0$ and $\ell\in\Z_+$}
such that for any   $\varphi, \psi \in  C^\infty (X) $
 and for any $t\ge 0$ one has
{\eq{em}{\left| {({g_t}\varphi ,\psi )- \int_X\varphi\, d\mu \int_X\psi\, d\mu}\right|   \ll  {e^{ - \gamma t}}{\left\| \varphi  \right\|_{\ell}} {\left\| \psi  \right\|_{\ell}}.}}
Here and hereafter ${\|\cdot \|_{\ell}}$  stands for the {\sl ``$L^{{2}}$, order $\ell$"} Sobolev norm, see 
{e.g.\ \cite[\S2]{KMi1} for a definition and basic facts}.

\smallskip
{The following is a special case of \cite[Theorem 1.1]{KMi1}:

\begin{thm}
\label{Main Theorem} Let $G$ be a Lie group, $\Gamma $   a uniform\footnote{This is a simplified version of the theorem; in  \cite{KMi1} instead of the compactness of $X$ it was assumed that the complement of $O$ was compact. The latter assumption has been removed in \cite{KMi2, KMi3} in many special cases.}
lattice in $G$, $X = G/\Gamma$, and let ${F^+}$ be a one-parameter {$\Ad$-}diagonalizable
{sub}semigroup of $G$ 
whose action on $X$ is exponentially mixing.
Then there {exists $ r_1 > 0$}
such that for any 
  $O\subset X$ 
  and
any
{$0 < r < 
 r_1
 $}
one has\footnote{The result in \cite{KMi1} actually involved the slightly smaller set $$E(F^+,O) := \{ x \in X: g_tx \notin
 O \ \forall\, t\ge 0\}
 $$ of points whose trajectories avoid $O$, bit it is easy to see that $\dim E(F^+,O) =  \dim \widetilde E(F^+,O)$.}
 \eq{mainbound}{ {
 \codim \widetilde E(F^+,O)  {\,\gg}\ }\frac{{\mu ({\sigma_{r}O)}}}{{\log (1/r) + \log \big(1/\mu ({\sigma_{r}O)\big)}}}\,.}
\end{thm}

An interesting feature of the above estimate is that its left hand side does not depend on $r$ while the right hand side does, and tends to $0$ when $r$ either tends to $0$ or becomes large enough. Thus for applications one is left to seek an optimal value of $r$. For example when $O = B(x,\rho)$ one can take $r = \rho/2$, producing the estimate \eq{balls}{ \codim \widetilde E\big(F^+,B(x,\rho)\big)  {\,\gg}\ \frac{\rho^{\dim X}}{{\log (1/\rho)}}\,.}
\ignore{Note also that the right hand side of  \equ{mainbound} is zero if $O$ has empty interior; and indeed it is known that for many non-empty sets $O$ the 
sets $\widetilde E(F^+,O)$ or orbits eventually avoiding $O$ can have full \hd\ \cite{K1, KW, AGK}. From now on we will assume $O$ to have non-empty interior, and will aim for non-trivial upper bounds for the \hd\ of various exceptional sets related to $O$.}
\smallskip

The main theme of the present paper is replacing the phenomenon of eventual escape by escape on average.
In order to state our main result, we need to introduce the   {function $\phi:[0,1]\times[0,1]\to\R$ by}
\eq{defphi}{
\phi(y,s):= ({ 1-s})  
{\log\frac1{1-\frac y2}}- 
{2({ 1-s})}   \log \frac{1} {{ 1-s}} - s \cdot \log \frac{{{1+\frac y4}}}{s}.            }
{We will employ the convention $0 \cdot\log\frac10 = 0$; this way we can see that $\phi$ is a continuous function on $\R^2$, with \eq{zerovalue}{\phi(y,0) = {\log\frac1{1-\frac y2}}\text{ for all }y,}}
{hence $\phi(y,s) > 0$ for small enough positive $s$ {as long as $y>0$}.}

\smallskip
 The following is {our first main theorem}:


{\begin{thm}\label{average}  {Let $G$, $\Gamma$, $X$ and $F^+$ be as  in Theorem \ref{Main Theorem}.}
Then there exists positive constant {$r_1$} such that for any 
$O\subset X$, any $0<r<{r_1}$ and  any $\delta>0$  one has
\eq{mainestimate}{ \codim E_{\delta}(F^+,O) \gg \frac{{\mu(\sigma_r O) \cdot \phi\big(\mu(\sigma_r O), \sqrt{1-\delta}\big)}}{\log \frac{1}{r}} .}
\end{thm}  
Thus, if given a subset $O$ of $X$ we define
\eq{fd}{\delta_O:=\inf \left\{ 0<\delta<1: \phi\big(\mu(O), \sqrt{1-\delta}\big)>0 \right\},}
{which is {strictly} less than $1$
whenever $\mu(O) > 0$, and  if {in addition $O$ has non-empty interior, {then} for any
$\delta > \delta_O$  one can choose $r>0$  small enough to  have $\phi\big(\mu(\sigma_r O), \sqrt{1-\delta}\big)>0$. 
This implies  \eq{maincor}{ \codim E_{\delta}(F^+,O) > 0\text{ whenever }\delta > \delta_O.}} }
{Also 
{one can   observe} that {$$\phi\left(y,\frac{y}{2}\right) = - 
{\left({ 1-\frac{y}{2}}\right)}   \log \frac{1} {{ 1-\frac{y}{2}}} - \frac{y}{2} \cdot \log \frac{{{1+\frac y4}}}{{y}/{2}}< 0$$ for any $0 < y < 1$; hence $ \delta_O > 1 - {{\frac{\mu(O)^2}{4}}} > \mu(O^c)$}.}

{As an example, here is a graph of the function {$\delta\mapsto \phi(0.8, \sqrt{1-\delta})$; thus for $\mu(O) = \frac45$ we get {$\delta_O \approx 0.9888$}.\bigskip }}

\centerline{ \includegraphics[scale=.5]{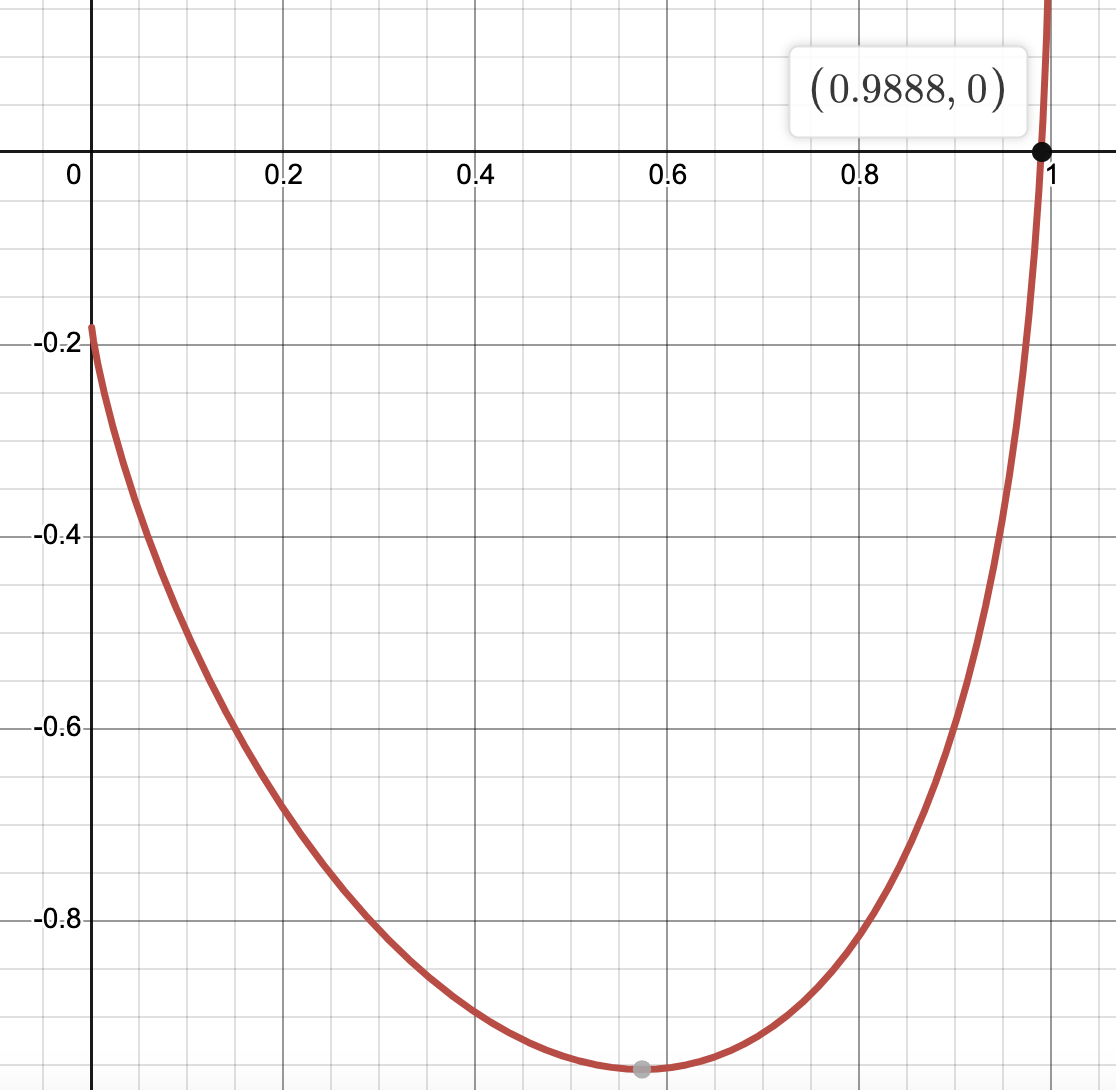}}

{\bigskip In the special case $\delta=1$, using \equ{zerovalue}} we obtain the following immediate corollary:
\begin{cor}\label{cor1}
Let $G$, $\Gamma$, $X$, $F^{+}$ and ${r_1}$ 
be as in Theorem \ref{average}. Then for any {$0<r<r_1$} we have
$$ \codim E_{1}(F^+,O) \gg  {\frac{{\mu(\sigma_r O)}\cdot\log\frac1{1-\frac {\mu(\sigma_r O)}2}} {\log \frac{1}{r}}} \gg \frac{{\mu(\sigma_r O)}^2}{\log \frac{1}{r}} .$$
\end{cor}
One sees that the above corollary does not produce any improvement {of} \equ{mainbound}; on the contrary, our new dimension bound for escape on average happens to be worse by a factor of $\mu(\sigma_r O)$ than the bound for the eventual escape. {In the sequel to the present paper, using some ideas from the work 
\cite{MRC} where a similar problem was studied for the Teichm\"uller geodesic flow, the authors plan to improve the existing estimates as well as extend the methods to treat non-compact \hs s.}

\smallskip
   {We also remark that when $X$ is not compact, one can consider the set 
      $$
      E^{comp}_\delta(F^+) := 
      \bigcap_{C \subset X \,\text{compact}}{E}_{\delta}(F^+,C) $$
      of 
      points in $X$ with trajectories $\delta$-escaping all compact subsets of $X$. Equivalently one can 
    define $E^{comp}_\delta(F^+)$ as the set of $x\in X$ such that  there exists a sequence $T_k \to\infty$ and a weak-$*$
limit $\nu$ of the sequence
of probability measures 
$f\mapsto\frac1
{T_k}\int_0^{T_k} f(g_tx)\,dt $
such that $\nu(X)\le 1 - \delta$. It is proved in  \cite[Theorem 1.3]{RHW} that  whenever $G$ is a connected semisimple Lie group, $\Gamma $   a 
lattice in $G$ and $F^+$ a  one-parameter semigroup contained in one of the simple factors of $G$, 
one has $$\codim\big(E^{comp}_\delta(F^+)\big) \ge c\delta\text{ for all }0 < \delta < 1,$$ where $c$ depends only on $G$ and $\Gamma$. See also 
 \cite[Remark 2.1]{KKLM} for a special case.}

\smallskip

From now on let $G$, $\Gamma$, $X$ and $F^+$  be as in {Theorems \ref{Main Theorem} and \ref{average}}.
Similarly to the argument in \cite{KMi1, KMi2, KMi3},
{our}
main theorem is deduced from
 {a result} that estimates
$$\dim E_{\delta}(F^+,O) \cap Hx,$$ where $x\in X$ and $H$ is the {\sl unstable horospherical subgroup} with respect to $F^+$, {defined as
\eq{uhs}{H := \{ g \in G:{\dist}({g_t}g{g_{ - t}},e) \to 0\,\,\,as\,\,\,t \to  - \infty \}.}}
{More generally, in the  theorem below we estimate $$\dim
E_\delta(F^+,O) \cap Px$$ for {$x\in X$ and} some proper
subgroups $P$ of $H$, namely those which have so-called  Effective
Equidistribution Property  (EEP) with respect to {the flow $(X,F^+)$. The latter property was {motivated by \cite{KM4} and} introduced in \cite{KMi1}, {where} it was shown to hold for $H$ as above under the assumption of exponential mixing. Namely, we have the following 

\begin{defn}\label{subgroup}
Say that a subgroup $P$ of $G$ has {\sl Effective Equidistribution
Property} {\rm (EEP)} with respect to the flow $(X,F^+)$
if 
 $P$ is normalized by $F^+$, and there exist 
$ 
\lambda > 0$, {$t_0 > 0$} 
and {$\ell \in {\N}$} such that 
for any 
{$x \in X$, {$t \ge t_0$}, 
 $f\in C^\infty(P)$ with $\supp f \subset B^P(1)$ and 
 $\psi\in
C^\infty(X)$}
it holds that
\eq{eep}{\left| {
\int_P
f({{h}})\psi(g_t{{h}}x)\,d\nu({{h}})
- \int_P f\,d\nu {\mkern 1mu} \int_X \psi\,d\mu  {\mkern 1mu} } \right| 
\ll {\max \big(
{\left\| \psi  \right\|_{C^1}, {\left\| \psi  \right\|_{\ell}} }
\big)} \cdot {\left\| f \right\|_{{C^\ell }}} \cdot {e^{ - \lambda t}}{\mkern 1mu}.}
\end{defn}

\ignore{To relate (EEP) with the theorem you applied in the MRC paper, just integrate \equ{eep} from $0$ to $T$, where $T$ is much bigger than $t_0$. For simplicity let me write it down with $f = 1|_V$, where $V\subset P$ is a bounded open set with smooth boundary. Then, after approximating $1|_V$ by smooth functions,  \equ{eep} looks like 
\eq{eepnew}{\left| {
\int_V
\psi(g_t{{h}}x)\,d\nu({{h}})
- \nu(V) {\mkern 1mu} \int_X \psi\,d\mu  {\mkern 1mu} } \right| 
\ll C(\psi,V){e^{ - \lambda t}}{\mkern 1mu},}
where $C(\psi)$ is some constant depending on $\psi$. Now 
average form $0$ to $T$:
\eq{eepnew}{\left| {
\frac1T\int_0^T\int_V
\psi(g_t{{h}}x)\,d\nu({{h}})dt
- \nu(V) {\mkern 1mu} \int_X \psi\,d\mu  {\mkern 1mu} } \right| 
\ll C'(\psi,V) \left(\frac{t_0}T+ \frac{Tt_0}T{e^{ - \lambda t}}\right){\mkern 1mu},}
}}

We remark that in  \cite{KMi1, KMi2, KMi3} this property was defined and used in the more general set-up of $X$ being non-compact, with additional constraints on the {\sl injectivity radius} 
at points {$x\in X$ for which \equ{eep} is satisfied} {(see \S\ref{r2} for more detail)}. However when $X$ is compact the injectivity radius is uniformly bounded from below, {hence a possibility to simplify the definition.}

\ignore{Here is an important non-horosherrical example in the non-compact setting: when $X =\SL_{m+n}(\R)/\SL_{m+n}(\Z)$ 
and
$$g_t = {g_t^{\vr,\vs} :=} \diag({e^{{i_1}t}},\dots,{e^{{i_m}t}},{e^{
- {j_1}t}},\dots,{e^{ - {j_n}t}}),$$
where
$$\vr = ({i_k}: k = 1,\dots, m)\text{ and }\vs = ({j_\ell}:  \ell = 1,\dots,  n)$$
{and 
$${{i_k},{j_\ell} > 0\,\,\,\,and\,\,\,\,\sum\limits_{k = 1}^m {{i_k} = 1 = \sum\limits_{\ell = 1}^n {{j_\ell}} },}$$}
it is shown in \cite{KMi1}
that the subgroup
$$
{P = \left\{ \left( {\begin{array}{*{20}{c}}
{{I_m}}&A\\
0&{{I_n}}
\end{array}} \right) : A\in  {M_{m,n}}(\mathbb{R})\right\}
}$$
satisfies (EEP) relative to the
$g_t^{\vr,\vs}$-action. This example is important for Diophantine applications, see \cite[Theorem 1.4]{KMi1}. \comm{Maybe move this remark away from the introduction?}}

\ignore{Let $\frak p$ be the Lie algebra of $P$.
Define
\eq{b1}{{{\lambda_{\min}}} := \min \{ |\lambda |:\,\lambda \text{ is an eigenvalue of }\ad_{{g_1}}|_{\frak p}\} }
and
\eq{lmax}{\lambda_{\max} := \max \{ |\lambda |:\,\lambda \text{ is an eigenvalue of }\ad_{{g_1}}|_{\frak p}\} }}


{{\begin{thm}\label{dimension drop 3} Let $G$, $\Gamma$, $X$, $F^{+}$   
be as above.
Then there exists ${{r_2} 
>0}$ such that whenever {a connected subgroup  $P$  of $H$} 
has property {\rm (EEP)} with respect to {the flow $(X,F^+)$},
for any non-empty open subset $O$ of $X$,  any $0<r\le 
{r_2}$,   any $0<\delta < 1$ and any $x\in X$ one has
{$$ \codim \big(E_{\delta}(F^+,O) \cap Px\big) 
\gg \frac{{\mu(\sigma_r O) \cdot \phi\big(\mu(\sigma_r O), \sqrt{1-\delta}\big)}}{ \log \frac{1}{r}}.$$} 
 \end{thm}}} }
{In the next section we    derive Theorem \ref{average} from Theorem \ref{dimension drop 3}, and 
the rest of the {paper is} be dedicated to proving Theorem \ref{dimension drop 3}. 
{Section \ref{aux} contains a discussion of basic technical constructions needed for the proof, such as \hd\ estimates for lim sup sets, tessellations of nilpotent Lie groups and Bowen boxes.
In Section \ref{covcpt} we use property (EEP) to, given a subset $S$ of $X$, a large $T> 0$ and $0 < \delta < 1$,  write down a measure estimate for the set of   $h\in P$
such that the Birkhoff average $\frac{1}{T} \int_0^{T} 1_{S}(g_thx)\,dt \ge \delta$. In the subsequent section this estimate is used to bound the number of Bowen boxes that can cover certain exceptional sets. Finally, Section  \ref{proof} contains the conclusion of the proof.}

\section{Theorem \ref{dimension drop 3} $\Rightarrow$ Theorem \ref{average}
}

The reduction of Theorem \ref{average} to Theorem \ref{dimension drop 3} is fairly standard. Let $\mathfrak{g}$ be a Lie algebra of $G$,
$\mathfrak{g}_\mathbb{C}$ its complexification, and for $\lambda
\in \mathbb{C}$, let $E_ \lambda$ be 
the eigenspace of
$\Ad\, g_1$ corresponding to $\lambda$.
Let $\mathfrak{h}$, $\mathfrak{{h^0}}$, $\mathfrak{{h^ - }}$ be the
subalgebras of $\mathfrak{g}$ with complexifications:
\[{\mathfrak{h}_\mathbb{C}} = \spn({E_\lambda }: \left| \lambda  \right| > 1),\ \mathfrak{h}_\mathbb{C}^0 = \spn({E_\lambda }: \left| \lambda  \right| = 1),\ \mathfrak{h}_\mathbb{C}^ -  = \spn({E_\lambda }:\left| \lambda  \right| < 1).\]
Let $H$, ${H^0}$, ${H^ - }$ be the corresponding subgroups of $G$. 
{Note that} $H$ is {precisely} the
unstable horospherical subgroup {with respect to $F^+$} (defined in \equ{uhs}) and $H^-$ is {the}
stable horospherical subgroup defined by:
\[{H^ - } = \{ h \in G:{g_t}h{g_{ - t}} \to e\,\,\,as\,\,t \to  + \infty \}. \]
Since $\Ad\, g_1$ is assumed to be diagonalizable over $\mathbb{C}$, $ \mathfrak{g} $ is the direct sum of
$\mathfrak{h}$, $\mathfrak{{h^0}}$ and $\mathfrak{{h^ - }}$. Therefore $G$ is
{locally (at a neighborhood of identity) a} direct product of the subgroups $H$, ${H^0}$ {and} ${H^ - }$.

{Denote the group ${H^ - }{H^0}$ by $\tilde H$, and fix ${0 < \rho < 1}$ with the following properties:}
{
\begin{equation}\label{rho1}
 \text{the multiplication map }\tilde H \times H \to G\text{ 
is one to one  on }B^{{\tilde H}}(\rho) \times {B^H}(\rho), \end{equation}
and
\begin{equation}\label{conjugate implied}
{g_tB^{\tilde H}(r)g_{-t} \subset B^{\tilde H}(2r) \text{ for any $0<r \le \rho$ and }t\ge 0}\end{equation}}
{(the latter can be done since  $F$ is $\Ad$-diagonalizable and  the restriction of the map $g \to g_tgg_{-t}$,
$t > 0$, to the subgroup $\tilde P$
is non-expanding).} 

\begin{proof}[Proof of Theorem \ref{average} assuming Theorem \ref{dimension drop 3}]
Let
{
$\rho$ be as in \eqref{rho1}, \eqref{conjugate implied} and define 
$$r_1:= \min(\rho, {r_2})
.$$
For any {$0< r< r_1$} choose $s$ such that 
\eq{cont1}{{B^G
}(s) \subset B^{{\tilde H}}(r/4){B^H}(r/4).}
} 
Now take   $O \subset X$, and for  $x\in X$ denote 
\eq{reduced set}{{{E_{\delta,x,s}}:=}\, \left\{ g \in {B^G
}(s):gx \in E_\delta(F^+,O)\right\}.}
%
In view of the countable stability of Hausdorff dimension, in order to prove the theorem it suffices to show that for any  $x \in X$,
\eq{need}{
\codim E_{\delta,x,s}  \gg  \frac{{\mu(\sigma_r O) \cdot \phi\big(\mu(\sigma_r O), \sqrt{1-\delta}\big)}}{\log \frac{1}{r}}  }
Indeed, $E_\delta(F^+,O)$ can be covered by countably many sets $\{gx : g \in E_{\delta,x,s}\}$, with the quotient maps 
{$\pi_x: B^G
(s) \to X$} 
being Lipschitz and one-to-one.  
Since every ${g\in B(s)}
$ can be written as $g = {\tilde h}h$, where
${\tilde h} \in {B^{\tilde H}}(r/4)$ and $h \in {B^{H}}(r/4)$, 
for any $y \in X$ we can write
\eq{transition1}{
\begin{aligned}
{\dist}({g_t}gx,y) &\le {\dist}({g_t}{\tilde h}hx,{g_t}hx) + {\dist}({g_t}hx,y)\\ &={\dist}\big(g_t{\tilde h}g_{-t}{g_t}hx,{g_t}hx\big)+ {\dist}({g_t}hx,y).\end{aligned}}
Hence in view of \eqref{conjugate implied} and \equ{cont1},
$g \in {E_{\delta,x,s}}$ implies that {$h{x}$ belongs to $E_\delta({F^ + },{\sigma
_{{r/2}}}O)$}, and by using
Wegmann's Product Theorem \cite{Weg} we conclude that: 
\eq{wegmann}
{\begin{split}
\dim {E_{\delta,x,s}} 
&\le {\dim} \left(\{ h \in {B^H}(r/4):hx \in E_\delta({F^ + },{\sigma
_{r/2}}O)\} \times {B^{\tilde H}(r/4)}
\right)\\
& \le {\dim} \big(\{ h \in {B^H}(r/4):hx \in E_\delta({F^ + },{\sigma
_{r/2}}O)\}\big) +  \dim  {\tilde H }.
\end{split}
}
Since $(X, F^+)$ is exponentially mixing, by \cite [Theorem 2.6]{KMi1} $H$ has (EEP) with respect to $(X,F^+)$. 
Therefore, {by Theorem \ref{dimension drop 3} applied to $O$ replaced by $\sigma_{r/2}O$} and $r$ replaced with $r/4$ we have for any $0<r< r_1$
{
\eq{estimateH}{\begin{split} 
& \codim \{ h \in {B^H}(r/4):hx \in E_\delta({F^ + },{\sigma
_{r/2}}O)\}  \\
& \gg  \frac{{\mu(\sigma_{r/4} \sigma_{r/2} O) \cdot \phi\big(\mu(\sigma_{r/4} \sigma_{r/2} O),\sqrt{1-\delta}\big)}}{\log \frac{4}{r}} \\
& \gg  \frac{{\mu(\sigma_r O) \cdot \phi\big(\mu(\sigma_r O), \sqrt{1-\delta}\big)}}{\log \frac{4}{r}} 
 \gg \frac{{\mu(\sigma_r O) \cdot \phi\big(\mu(\sigma_r O), \sqrt{1-\delta}\big)}}{\log \frac{1}{r}}.
\end{split}}
Now from \equ{wegmann} and \equ{estimateH} we conclude that \equ{need} is satisfied, which finishes the proof. }
\end{proof}

\ignore{\begin{proof}[Proof of Corollary \ref{balls}] {If $S = \varnothing$ there is nothing to prove. Otherwise, by   Theorem \ref{average} 
{applied to $U = \partial_rS$ and  with $r/2$ in place
of $r$}, any $0<r<r_0$, any $s>0$, and for any $0<\delta<1$ the inequality \equ{rnei} holds. Similarly \equ{rnei1} follows from Corollary \ref{cor1} applied to $U=\partial_r S$.
 This proves the main part of the corollary.}
 \smallskip
 \\
{For the ``consequently" part, if $S$ is a $k$-dimensional compact embedded submanifold in $X$, then it is
easy to see that 
one can find {$0<r_S<r_0$}
dependent on $S$ {such that  for all $r< 
r_S
$ one has}  \eq{partial 3}{\mu ({\partial _{r/2}}S)\, {\ge} \, \frac{1}{2}{\left(\frac{r}{2}\right)^{\dim X
- k}}.}
So by setting $s=\frac{1}{4}{\left(\frac{r}{2}\right)^{\dim X
- k}}$  and in view of \equ{partial 3}, consequently part follows form the main part of the corollary. }
\end{proof}}

\section{Auxillary facts}\label{aux}
\subsection{\hd\ of limsup sets}
    The exceptional sets we study in this paper are of the form $A=\limsup\limits_{N\r\infty} A_N$, that is 
    \[A=\bigcap\limits_{N\geq 1}\bigcup\limits_{n\geq N}A_n\]
    for a sequence of subsets $A_N$.
    
    First, we recall the definition of the Hausdorff dimension.    
    Let $A$ be a subset of a metric space $Y$. For any  $\rho ,\beta >0$, we define
        \[ {\mathcal H}_\rho^\b (A) = \inf\set{ \sum_{I\in \mc{U}} \diam(I)^\beta: \mc{U} \text{ is a cover of A by balls of diameter } <\rho  }. \]
        
        Then, the $\beta$-dimensional Hausdorff measure of $A$ is defined to be
        \[ {\mathcal H}^\b (A) = \lim_{\rho \r 0} {\mathcal H}_\rho^\b(A). \]
 
 \begin{defn}\label{defhd}
        The Hausdorff dimension of a subset $A$ of a metric space $Y$ is equal to
        \[ \dim
        (A) = \inf\set{\b\geq 0: {\mathcal H}^\b(A) = 0}= \sup \set{\b\geq 0: {\mathcal H}^\b(A) = \infty}. \]
 \end{defn}
\begin{lem} \label{Auxillary}
        Let $\set{A_N}_{N\geq 1}$ be a collection of subsets of $Y$.
        Suppose there exist constants {$C  >0$, $0 < \rho < 1$, $N_0\in\N$}, and a sequence  $\{\alpha_N\}_{N\ge 1}$ such that for each $N\ge N_0$, the set $A_N$ can be covered with at most {$\rho^{-\alpha_N N}$
        balls of radius $C 
        \rho^N$}. {Then $$\dim\left(\limsup\limits_{N \r\infty} A_N\right) \le \liminf\limits_{N \r\infty} \alpha_N.$$}
      \end{lem}
      
      \begin{proof}
      	Let {$A:=\limsup\limits_{N \to\infty} A_N$ and $\a:= \liminf\limits_{N \to\infty} \alpha_N$; without loss of generality we can assume that $\alpha < \infty$ and $\alpha_N\to\alpha$ as $N\to\infty$.
	Take} $\b 
	> \a$; 
         we will show that ${\mathcal H}^\b(A) = 0$, which will imply the lemma.
         {For any ${{\xi}} \in (0,1)$, let ${N_\xi\ge {N_0}}$ be a natural number such that $\rho^{-N } < {{\xi}} $ {and $|\alpha_N- \alpha|< \frac{\beta - \alpha}{2}$} for all $N \geq {N_\xi}$. Notice that ${N_\xi}$ tends to infinity as ${{\xi}}$ goes to $0$.
         Take $N\ge N_\xi$ and denote by $\mc{O}_N$ a cover of the set $A_N$ by balls of radius ${C} 
        \rho^N$ such that $\#\mc{O}_N$, the number of balls in the cover $\mc{O}_N$, is  at most  ${\rho^{-{\alpha_N} N}}$. Then  $\mc{O}=\bigcup_{N \geq {N_\xi}} \mc{O}_N$ is a cover of $A$ for which the following holds:
         $$
   { \begin{aligned}
         	\sum_{B \in \mc{O}} \diam(B)^\beta 
         	& = \sum_{N \geq {N_\xi}} \sum_{B \in \mc{O}_N} \diam(B)^\b 
            	\le (2C)^\b \sum_{N \geq {N_\xi}} \# \mc{O}_N \cdot \rho^{\b  N} \\
         &      \leq (2C)^\b  \sum_{N \geq {N_\xi}}  \rho^{\left(\b -\a_N \right) N}
             \leq (2C)^\b  \sum_{N \geq {N_\xi}}  \rho^{\frac{\b -\a}{2} N} \\
          &\leq  (2C)^\b  \frac{\rho^{\frac{\b -\a}{2}{N_\xi} }}{1-\rho^{\frac{\b -\a}{2}}}
        \xrightarrow{{{\xi}}\r 0} 0.          \end{aligned}}
         $$}
        
        
        This implies that ${\mathcal H}^\b(A) = 0$, {and the conclusion of the lemma follows}.
      \end{proof}
\smallskip

\subsection{Choosing   {$r_2$}}\label{r2}
Recall that as part of the proof of Theorem \ref{dimension drop 3} we need to define a bound $r_2$ for possible values of $r$. This bound will come from two ingredients. Namely, we define
\eq{rlimit}{r_2:=\frac{1}{2} \min (r_0, r'),} 
where
\begin{itemize} 
\item   $0<r'<1/4$ is chosen so that for any Lie subalgebra  $\frak p$ of $\frak g$ 
the exponential map from  $\frak p$ to $P = \exp(\frak p)$ is $2$-bi-Lipischitz on $
 B^{\frak p}(r') $; in particular, we will have \eq{bilip}{B^P(2r)\supset \exp\big(B^{\frak p}(r)\big)\supset B^P(r/2)\text{  for any }0 < r \le r'.}
 
\item 
{${r_0:=r_0(X)=\inf \{r_0(x)  : x\in X 
\}>0,}$
where 
$${r_0(x) :=}\,\sup\{r > 0: 
\text{the map }\pi_x:g\mapsto gx\text{ is injective on }B(r)\}.$$
(the \textsl{injectivity radius} at $x_0$).}
\end{itemize}

\subsection{Tessellations of $P$}
Now let $P$ be a {connected} subgroup of $H$ normalized by $F^+$. Following \cite{KM1}, 
say that an open subset $V$ of $P$ is a {\sl tessellation domain} for $P$ relative to a countable subset $\Lambda$ of $P$ if 
\begin{itemize}
\item
$\nu (\partial V) = 0$;
\item
$V \gamma_1 \cap V \gamma_2 = \varnothing$ for different $\gamma_1,\gamma_2 \in \Lambda$;
\item
$P = \bigcup\limits_{\gamma  \in \Lambda } {\overline V \gamma }.$
\end{itemize} 

{Note that $P$ is a connected   simply connected nilpotent Lie group. Denote $\frak p := \Lie(P)$  and $L:= \dim P$, and fix a Haar measure $\nu$ on $P$.
As shown in  \cite[Proposition 3.3]{KM1}, one can choose a basis of $\frak p$ such that for any $r  > 0$, 
{$\exp \left(rI_{\frak p}\right)$, where $I_{\frak p} \subset \frak p $ is the cube centered at $0$ with side length $1$ with respect to that basis, is a tessellation domain. Let us denote
{\eq{Vr}{V_r :=\exp \left({{\frac{r}{4 \sqrt{L}}}} 
{I_{\frak p}}\right)}} 
and choose  a countable $\Lambda_r\subset P$   such that $V_r$  is a tessellation domain 
relative  to $\Lambda_r$.}
{
Then it follows from \equ{bilip} that for any $0<r \le r'$ {one has}
{\eq{Bowen inc}{{B^P}\Big(\frac{r}{{16  \sqrt{L} }}\Big) \subset {V_r} \subset {B^P}\left(\frac r {4}\right).}}


\subsection{Bowen boxes}
 Note that the measure $\nu$ and the pushforward of the Lebesgue measure $\Leb$ on $\frak p$ {by the exponential map} are absolutely continuous with respect to each other with locally bounded Radon--Nikodym derivative. This implies that there exists $0< c_1< c_2$ such that 
\eq{lb1}{
c_1\Leb(A) \le \nu\big(\exp(A)\big)\le c_2\Leb(A)\quad \forall\, \text{measurable } A\subset B^{\frak p}(1).
}

\ignore{Applying a linear time change to the flow $g_t$, without {loss of}  generality we can assume that 
Define
\gr
{\eq{eigen value}
{\lambda_{\max} = 1.}}}
Define
\eq{b1}{{{\lambda_{\min}}} := \min \{ |\lambda |:\,\lambda \text{ is an eigenvalue of }\ad_{{g_1}}|_{\frak p}\} }
{and}
\eq{lambda''}{{{\lambda_{\max}}}:=\max \{ |\lambda |:\,\lambda \text{ is an eigenvalue of }\ad_{{g_1}}|_{\frak p}\}.}

{Using the  bi-Lipschitz property of $\exp$, we can conclude that}
for any $0<r \le r'$ and any $t>0$ one has
\eq{diam}{\diam(g_{-t}V_rg_t)<  2r    e^{-{{\lambda_{\min}}} t} 
.}
{Also let 
$\eta := \Tr\ad_{{g_1}}|_{\frak p}$;
clearly one then has} \eq{delta}{{\nu(g_{-t}A g_t) = e^{-\eta t}\nu(A)\text{ for any measurable }A\subset P.}}
{Let us now define} a {\sl Bowen $(t,r)$-box} in $P$ 
{to be} a set of the form $g_{-t}V_r \gamma g_t$  for some $\gamma \in P$ {and $t > 0$}. 
{Our approach to estimating \hd\ of various subsets of $P$ will be through covering them by Bowen boxes. 
We are going to need {three} results proved 
in \cite{KMi3}. The first one, a slight modification of \cite[Proposition 3.4]{KM1}}, gives   an upper bound for 
the number of $\gamma  \in \Lambda_r$ such that the Bowen box $g_{-t}V_r \gamma g_t$ has non-empty intersection with $V_r$:

\begin{lem} 
\label{covering} \cite[Lemma 4.1]{KMi3}There exists $c_0 > 0$ such that for any $0<r \le r'$ and for any {$t>\frac{\log 8}{{{\lambda_{\min}}}} $}  
\[\# \{ \gamma  \in \Lambda_r :{g_{ - t}}{\overline{V_r}\gamma}{g_t}  \cap {V_r} \ne \varnothing \}  \le e^{\eta t}  \left(1 + c_0e^{-{{\lambda_{\min}}} t}\right).\]
\end{lem}

{The second one gives us an upper bound for the number of balls of radius $re^{- {{\lambda_{\max}}} t}$ needed to cover a Bowen {box}. 
\begin{lem}\cite[Lemma 7.4]{KMi3}
\label{coveringballs} There exists $C_1 > 0$ such that for any $0<r <1$ and any {$t>0 $},  any Bowen $(t,r)$-box in $P$ can be covered with at most  $C_1 e^{({{\lambda_{\max}}} L- \eta)t}$
balls in $P$ of radius $re^{- {{\lambda_{\max}}} t}$.
\end{lem}}

{The {third} result is a direct consequence of property (EEP); it is a simplified version of \cite[Proposition {4.4}]{KMi3}.}

\begin{prop} \label{exponential mixing} {Let ${F^+}$ be a one-parameter 
{sub}semigroup of $G$ and $P$ a subgroup of $G$ with property {\rm (EEP)} with respect to $F^+$. Then there exist 
{$t_1\ge 1$ and ${\lambda '
>0}$} such that
for any open $O \subset X$,  
any $x\in X$,  
any $r \le r_2$
and any $t \ge t_1$ one has
{
$${ {\mathop  \nu \big(\{ h \in V_r: g_thx \in O\} \big)\ge \nu\left(V_r\right)\mu ({\sigma _{e^{ -{ \lambda '} t}}}O)  - {e^{ - {{\lambda '} t}}}.}}$$}}
\end{prop}

\section{Effective equidistribution of translates and a measure estimate }\label{covcpt}

{From now on we will work with $F^+$ and $P$ as in Theorem \ref{dimension drop 3}, and for the rest of the paper fix a positive $r \le r_2$, where $r_2$ is
as in \equ{rlimit}}.  {Note that by the countable stability of Hausdorff dimension, in order to prove Theorem \ref{dimension drop 3} it suffices, for a subset $O$ of $X$ and $0 < \delta< 1$, to get a uniform (in $x\in X$) upper bound for the Hausdorff dimension of the set \eq{S}{\begin{aligned}{S_{x, \delta }(O)}&:=\{ h \in V_r: hx \in E_\delta(F^+, O) \} \\ & =\left\{h \in {V_r}: \limsup_{T \rightarrow \infty} \frac{1}{T} \int_0^T 1_{O^c}(g_thx)\,dt \ge \delta\right\}.\end{aligned}}
For this it will be convenient to discretize the above definition. Namely let us introduce the following notation:  given $T>0$, a subset $S$ of $X$, $x \in X$ and  $0<\delta <1$, let us define 
\eq{A}{A_{x, \delta}(T,S):= \left \{h \in V_r: \frac{1}{T} \int_0^{T} 1_{S}(g_thx)\,dt \ge \delta \right \}.          }
In the following proposition we find the relation between the set $S_{x, \delta }(O)$ and the family of sets $\{A_{x, \delta}(NT,O^c)\}_{N \in \N}$. 
\begin{prop}\label{lsup} 
For any $T>0$ and any $x \in X$ and we have
$${S_{x, \delta }(O)= \limsup_{N\in\N,\,N \rightarrow \infty} A_{x, \delta}(NT,O^c)}$$
\end{prop}
\begin{proof}
In view of definition of $S_{x,\delta}(O)$, it suffices, {for a fixed $T>0$,}  to prove that
$$\begin{aligned} \limsup_{R \rightarrow \infty} \frac{1}{R} \int_0^R 1_{O^c}(g_thx)\,dt&=\limsup_{N\in\N,\,N \rightarrow \infty} \frac{1}{NT} \int_0^{NT} 1_{O^c}(g_thx)\,dt\\ 
& = \lim_{N_0 \rightarrow \infty}\sup_{N\ge N_0} \frac{1}{NT} \int_0^{NT} 1_{O^c}(g_thx)\,dt.\end{aligned} $$
\ignore{By definition of $\lim \, \sup$ we have
\begin{align*}
& \lim \sup_{S \rightarrow \infty} \frac{1}{S} \int_0^S 1_{O^c}(g_thx)\,dt   \\
& = \lim_{T_0 \rightarrow \infty} \underset{S \ge T_0}{\sup}\frac{1}{S} \int_0^S 1_{O^c}(g_thx)\,dt \\
& = \lim_{N_0 \rightarrow \infty} \underset{S \ge N_0T}{\sup}\frac{1}{S} \int_0^S 1_{O^c}(g_thx)\,dt
\end{align*}
So it is enough to prove that
$$ \lim_{N_0 \rightarrow \infty} \underset{N \ge N_0}{\sup}\frac{1}{NT} \int_0^{NT} 1_{O^c}(g_thx)\,dt  =   \lim_{N_0 \rightarrow \infty} \underset{S \ge N_0T}{\sup}\frac{1}{S} \int_0^S 1_{O^c}(g_thx)\,dt.       $$
Clearly, we have
$$    \lim_{N_0 \rightarrow \infty} \underset{N \ge N_0}{\sup}\frac{1}{NT} \int_0^{NT} 1_{O^c}(g_thx)\,dt  \le   \lim_{N_0 \rightarrow \infty} \underset{S \ge N_0T}{\sup}\frac{1}{S} \int_0^S 1_{O^c}(g_thx)\,dt          $$
Thus, we just need to show the reverse inequality.}
 Let $N_0 \in \N$, and let $R \ge N_0T$. Then one can find $N \ge N_0$ and $0\le R' < T$ such that $R=NT+ R'$. Hence
$$
\begin{aligned}
 \frac{1}{R} \int_0^R 1_{O^c}(g_thx)\,dt 
 =&\frac{1}{NT+ R'} \int_0^{NT+ R'} 1_{O^c}(g_thx)\,dt   \\
 \le& \frac{1}{NT} \int_0^{NT+ R'} 1_{O^c}(g_thx)\,dt \\
 = & \frac{1}{NT} \left( \int_0^{NT} 1_{O^c}(g_thx)\,dt+  \int_{NT}^{NT+R'} 1_{O^c}(g_thx)\,dt \right) \\
 \underset{(1_{O^c} \le 1,\, R'  \le T)}{\le}&\frac{1}{NT}  \int_0^{NT} 1_{O^c}(g_thx)\,dt + \frac{1}{N} \\
 \underset{(N \ge N_0)}  {\le } &\frac{1}{NT} \int_0^{NT} 1_{O^c}(g_thx)\,dt + \frac{1}{N_0} .
\end{aligned}$$
Therefore, we get
$$  
 \limsup_{R \rightarrow \infty} \frac{1}{R} \int_0^R 1_{O^c}(g_thx)\,dt \le   \lim_{N_0 \rightarrow \infty} \underset{N \ge N_0}{\sup}\frac{1}{NT} \int_0^{NT} 1_{O^c}(g_thx)\,dt.       $$
{The reverse inequality is obvious.}
\end{proof}}

\ignore{Our goal in this section is to prove   a lower bound for the measure of $h\in V_r$ such that, {for an open subset $O$ of $X$, $x\in X$, $0 < \delta < 1$   and 
$T>0$,}  the proportion of $t\in [0,T]$ that the $g_t$-orbit of $hx$ spends outside of $O$ is at least $\delta$. {More precisely, for a subset $S\subset X$ let us define
$$A_{x, {\delta}}
 (T, S):=\left\{h \in {V_r}: \frac{1}{T} \int_{0}^{iT} 1_{S}(g_thx)\,dt \ge {\delta} 
  \right \}.$$}}}

\ignore{Since we are assuming that $X$ is compact, 
the \textit{injectivity radius} of $X$. which is defined as follows, is positive:
$$r_0:= r_0(X)= \sup\{r > 0: 
\pi_x\text{ is injective on }B(r)\  \ \forall\,x\in K\},$$
where for any $x \in X$, $\pi_x$ is the map $G\to X$ given by $\pi_x(g) := gx$. }

\smallskip
The next proposition gives {an} upper estimate for the measure of 
{$A_{x, {\delta}}(T,O^c)$.  We {will} use it  
in \S \ref{cover} to obtain an upper bound for the number of Bowen $(r,T)$-boxes in $P$ needed to cover 
this set} (see Corollary \ref{sub count}). 

\begin{prop} \label{exponential mixing1} 
For all 
{$x \in X$,  $0<\vre <1$,  for 
any {open} $O \subset X$ 
and for all}
{\eq{Tr}{T >T_r:= \max \left(t_1,\, \frac{1}{\l'} \log \frac{2}{r},\, \frac{1}{\l'} \log \frac{(4 \sqrt{L})^L}{ c_1 \l' r^L} {, \frac{\log 8}{{{\lambda_{\min}}}}} \right),}} where {$c_1$ is as in \equ{lb1}} and $t_1,\l'$ 
are as in Proposition \ref{exponential mixing},
one has
\eq{main}{\begin{aligned}
  \nu \big(A_{x, 1-\vre
 }(T,
  O)             \big)
  \ge 
 \nu\left(V_r\right)\left(1 -    {\frac{1}{\vre}} 
    \Big( 1-\mu ({\sigma _{r/2}}O) +\frac{{T_r+1}}{T} \Big)\right).
 \end{aligned}
}
\ignore{Here
\eq{T1}{T_1:=\max \left(a,\frac{1}{\lambda '} \log \left({
M_{\ell} {M'_{\ell}} {{E}}} +{c_2}2^L \right),\frac{\log \frac{1}{\lambda ' c_1}}{\lambda'} \right) +  \max \left(b, \frac{L}{\lambda '} \right) \log \frac{1}{r_0(X)}, }}

\ignore{
{\eq{conclusion}{ {\mathop {\inf }\limits_{x \in {\partial _{r}}{U^c}} \nu \left(\left \{h \in B^P \left( \frac{r}{16 \sqrt{L}}\right) : g_thx \in O^c \right \}\right)\ge \nu\left(B^P\Big(\frac{r}{16 \sqrt{L}} \Big)\right)\mu ({\sigma _{r}}U)  - E{e^{ - \lambda 't}},}}}
where $a'=a+\frac{\log 2}{\lambda '}, b'= \frac{1}{\lambda'},\lambda' = \frac{1}{2}\min(b, \frac{\lambda}{1+2 \ell}),$ and $E \ge 1$ only depends on $C, \ell,$ and $P$.
\bigskip}
\end{prop}
\begin{proof}
Let {$x,\vre,O$ and $T$ be as above.}
Then
by definition we have:
\eq{intermediate}
{\begin{aligned}
\nu \big(A_{x,\, 1-\vre}(T,O)             \big)
& =  \nu \left(\left\{h \in V_r: \frac{1}{T} \int_{0}^{T} 1_{{O}}(g_thx)\,dt \ge 1- \vre   \right \}             \right)     \\
&= \nu \left( V_r \right)-\nu \left(\left\{h \in V_r: \frac{1}{T} \int_{0}^{T} 1_{{O^c}}(g_thx)\,dt \,{>}\, \vre   \right \}             \right) \\
& {\ge  \nu \left( V_r \right)-\nu \big( A_{x, 
  {\vre} }(T,
  O^c)             \big)}.
\end{aligned}}
Our goal is to estimate the right-hand side of 
{\equ{intermediate} from below}. We have: 
{
\begin{equation}\label{e1}
\begin{aligned}
 \nu \big(A_{x,\vre
 }(T,
  O^c)             \big) = \  & \nu \left(\left\{h \in V_r: \frac{1}{T} \int_{0}^{T} 1_{{O^c}}(g_thx)\,dt \ge \vre   \right \}             \right) \\
  \underset{\text{(Markov's inequality) }}{\le} &
 {\frac1{\vre T}} \int_{0}^{T} \int_{ V_r} 1_{O^c}(g_thx) \,d \nu (h)\,dt \\
  =
& {\frac1{\vre T}} \int_{0}^{T} \nu \left(\left \{h \in V_r : g_thx \in O^c \right \}  \right) 
dt  \\
   =
 &{\frac1{\vre T}} \left( \int_{0}^{T_r} \nu \left(\left \{h \in V_r : g_thx \in O^c \right \}  \right) dt   \right.   \\ 
 &\qquad\qquad
 + \left.\int_{T_r}^{T } \nu  \left(\left \{h \in V_r : g_thx \in O^c \right \}  \right) dt   \right) .
 \end{aligned}
 \end{equation}}
 Note that 
{\eq{1inneq} 
{ 
 { \int_{0}^{T_r} \nu \left(\left \{h \in V_r : g_thx \in O^c \right \}  \right) dt  
  \le 
  T_r \cdot \nu\left(V_r\right),}      
   }}
 {and, since} 
$T_r \ge t_1$,
by applying Proposition \ref{exponential mixing} we get
\eq{2inneq}{ 
\begin{aligned}
  &\int_{T_r}^{T}  \nu \left(\left \{h \in V_r : g_thx \in O^c \right \}  \right)   \,dt  \\
= &\int_{T_r}^{T} \big(\nu(V) -  \nu \left(\left \{h \in V_r : g_thx \in O \right \}  \right) \big)  \,dt  \\
 \le 
&\int_{T_r}^{T}   \left( \nu \left( {V_r} \right) 
\left(1-\mu ({\sigma _{e^{ -  \lambda '  T_r}}}O)\right) + {e^{ -  \lambda ' t}}\right)   \,dt         \\
 \le 
&\ T\cdot \nu\left(V_r\right)
 \left( 1-\mu ({\sigma _{e^{-\l' T_r}}}O)\right) + \frac{e^{-\l' T_r}}{ \lambda '  
} 
\\
 \underset{\equ{Tr}}{\le}&\ { 
T\cdot \nu\left(V_r\right) 
\left(  1-\mu ({\sigma _{r/2}}O) \right) +  \frac{e^{-\l' T_r}}{ \lambda ' } 
}.
\end{aligned}}
{It remains to observe that $\equ{Tr}$, in combination with \equ{Vr} and \equ{lb1}, also implies that  $\frac{e^{-\l' T_r}}{ \lambda ' } \le \nu(V_r)$. Hence} \equ{main} follows from 
\equ{intermediate}, 
\eqref{e1}, \equ{1inneq} and \equ{2inneq}. 
\end{proof}

\section{A covering result}\label{cover}
In this section we will 
prove a covering result for the sets of the form $A_{x, \delta}(NT,O^c)$. Then, by using Proposition \ref{lsup} and Lemma \ref{Auxillary}, in the next section we will 
obtain an upper estimate for the Hausdorff dimension of ${S_{x, \delta }(O)}$.  
\smallskip

We start with the following lemma:
\begin{lem}
\label{covering of A^P} Let $O$ be an open subset of $X$, 
and let $0<\vre<1$.
    {Then for any  {${T>0}$}, $\gamma \in \Lambda_r$, $x \in X$, {any $0< \alpha< \vre$} and for any Bowen $(r,T)$-box $B={g_{ - T}}{V_r\gamma}{g_T}  $ which has non-empty intersection {with} the set {$A_{x, 1-\alpha}(T, \sigma_{4r} O)$}, we have {$$B  \cap  \, A_{x, \vre}(T,O^c )= \varnothing.$$}}
\end{lem}

\begin{proof}
Let $O$ be an open subset of $X$ and let 
$\gamma \in \Lambda_r$.  Consider the Bowen $(r,T)$-box $B={g_{ - T}}{V_r\gamma}{g_T}$.
Let $x \in X$,
take $p \in B$, and assume that {$g_tpx \in \sigma_{4r}O$} for some $0 \le t \le T$. {Any $p' \in B$ is of the form $p' = hp$ where $h\in {g_{ - T}}({V_r \cdot V_r}){g_T}$. Thus we have
$$
\begin{aligned}
g_tp'x 
= g_thg_{-t} g_tpx & \in g_{ - (T-t)}({V_r \cdot V_r})g_{(T-t)}g_tpx \\
& {\underset{\equ{diam}}{\in} B^G (2r e^{-{{\lambda_{\min}}}(T-t)} ) \cdot B^G (2r e^{-{{\lambda_{\min}}}(T-t)} )g_tpx} \\
& {\,\, \in B^G(4r) g_tpx}
\end{aligned}
$$
which, in view of \equ{Bowen inc}, implies that $\dist(g_tp'x , g_tpx ) \le {4r}$. Hence  $g_tp'x \in O$. Now assume in addition that  {$0<\alpha < \vre< 1$ and $p\in A_{x, 1-\alpha}(T, \sigma_{4r} O)$}; then}
 $$
 \begin{aligned}
  \frac{1}{T}\int_0^T 1_{ O^c}(g_tp'x)\,dt & =1- \frac{1}{T}\int_0^T 1_{ O}(g_tp'x)\,dt 
\\
 & \le 1- \frac{1}{T}\int_0^T {1_{\sigma_{4r}O}(g_tpx)}\,dt  
\le 1-(1-\alpha)=\alpha <  \vre.
 \end{aligned}$$
Therefore, {if $B $ has non-empty intersection with 
$A_{x, 1-\alpha}(T, \sigma_{4r} O) $}, then for any $p' \in B$ we have {$p' \notin  
A_{x, \vre}(T,O^c )$}. 
This ends the proof.  
\end{proof}

Now, by combining the previous lemma with Lemma  \ref{covering} and {{Proposition} \ref{exponential mixing1}} we obtain the following corollary which is a covering result for the sets of type $A_{x, \vre}(T,O^c)$:  
\begin{cor}\label{sub count}
Let 
$T_r$ be as in \equ{Tr}. Then for any $0<\vre<1$, any {$T>T_r$}, any $x \in X$, and for any open subset $O$ of $X$, the set $A_{x, \vre}(T,O^c)$ can be covered with at most
$ {\frac{e^{\eta T}}\vre C(T,O)}$     
Bowen $(T,r)$-{boxes} in $P$, where
\eq{ctoe}{{C(T,O):=   
 1-{\mu ({\sigma _{5r}}O)} +\frac{T_r+1}{T} 
 +c_0 e^{-{{\lambda_{\min}}} T}  
 .        }}
\end{cor}


 \begin{proof}
Let $0<\vre<1$, {$0< \alpha < \vre$}, {$T>T_r$}, $x \in X$, and let $O$ be an open subset of $X$. Then we have:
{
 \begin{align*}
     &\ \#\{\gamma \in \Lambda_r : g_{-T}V_r \gamma g_{T} \, \cap \,  A_{x, \vre}(T,O^c)  \ne \varnothing\} \\
\underset{\text{Lemma } \ref{covering of A^P}}\le  & \  \#\{\gamma \in \Lambda_r : g_{-T}V_r \gamma g_{T} \, \cap \,  V_r  \ne \varnothing\} -  \#\{\gamma \in \Lambda_r : g_{-T}V_r \gamma g_{T} \, \subset \, A_{x, 1-\alpha}(T, \sigma_{4r} O )  \} \\ 
 \underset{\text{Lemma } \ref{covering}}{\le}  &\ e^{\eta T}  \left(1 + c_0e^{-{{\lambda_{\min}}} T}\right) -    \#\{\gamma \in \Lambda_r : g_{-T}V_r \gamma g_{T} \, \subset \, A_{x, 1- \alpha}(T,\sigma_{4r} O )  \}   
 \\ {\le} & \  e^{\eta T}  \left(1 + c_0e^{-{{\lambda_{\min}}} T}\right) -\frac{\nu\left(A_{x, 1-\alpha}(T,\sigma_{4r} O )\right)}{\nu({g_{ -T}}V_r {g_T})} \\
 \underset{\text{Proposition }\ref{exponential mixing1}}{\le}&\ e^{\eta T}  \left(1 + c_0e^{-{{\lambda_{\min}}} T}\right)  -e^{\eta t} \left( 1-\alpha^{-1} \cdot \left( 1-\mu (\sigma_{r/2}{\sigma _{4r}}O) + {\frac{T_r+1}{T}}  \right)  \right)                      \\
 \underset{\equ{ctoe}}{<} & \ 
  {\frac{e^{\eta T}}\alpha C(T,O)}.  \end{align*} 
Now since $0< \alpha < \vre$ was arbitrary, by letting $\alpha$ approach $\vre$, we get that $A_{x ,\vre}(T, O^c)$ can be covered with at most $  {\frac{e^{\eta T}}\vre C(T,O)}$     
Bowen $(T,r)$-{boxes} in $P$, as desired. }  
 \end{proof}

Next we will need a generalized version of the definition \equ{A} of sets $A_{x, \delta}(T,S)$. 
Namely, given $S \subset X$, $x \in X$,   $T>0$, $0<\delta < 1$  and $J \subset \N$, let us define
\eq{atsj}{A_{x,\, \delta}
 (T,S,J):=\left\{h \in V_r: \frac{1}{T} \int_{(i-1)T}^{iT} 1_{S}(g_thx)\,dt \ge \delta  \,\,\,\,\, \forall\, i \in J \right \}.}
Clearly $A_{x, \delta}(T,S) = A_{x,\, \delta}
 (T,S,\{1\})$.
 Using the above corollary inductively, in the following proposition we obtain a covering result for the sets of the form 
 \equ{atsj}.

\begin{prop}\label{delta}
Let $O$ be a non-empty open subset of $X$, 
and let $T_r$ be as in \equ{Tr}. Then for all $0< \vre < 1$, {$T>T_r$}, $N \in \Z_+$, {$J \subset \{1,\dots,N\}$}, and for all $x\in X$, 
the set $A_{x, \vre}(T, O^c, 
J)$ can be covered with at most 
\eq{uppcov}
{e^{\eta NT} { { {\left(\frac{C(T,O)}\vre\right)^{|J|}}} \left(1 + c_0e^{-{{\lambda_{\min}}} t}\right)^{N- |J|}}}
Bowen $(NT,r)$-{boxes} in $P$.
\end{prop}
\begin{proof} Let $0< \vre <1$, let {$T> T_r$}, and let $x \in X$. {We argue by induction on $N$; the basis is given by $N=0$ and $J = \varnothing$, which makes the quantity \equ{uppcov} equal to $1$. This makes sense, since $A_{x, \vre}(T, O^c, 
J) = V_r$, which is precisely a $(0,r)$-box.} 

{Now take an arbitrary $N\in\N$ and let $J' := J \smallsetminus \{N\}$. By the induction assumption, 
the set 
$A_{x, \vre}\big((N-1)T,O^c,J'\big)$ can be covered with at most
\eq{uppcovprime}
{e^{\eta (N-1) T}\cdot {\left(\frac{C(T,O)}\vre\right)^{|J'|}} \left(1 + c_0e^{-{{\lambda_{\min}}} t}\right)^{N-1- |J'|}}
Bowen $\big((N-1)T,r\big)$-boxes in $P$.
Now let $g_{-(N-1)T}V_r \gamma g_{(N-1)T}$ be one of the Bowen $\big((N-1)T,r\big)$-boxes in the above cover which has non-empty intersection with $A_{x, \vre}\big((N-1)T,O^c,J'\big)$. 
Take any $q={g_{ - (N-1)T}}{h} \gamma{g_{(N-1)T}} \in g_{-(N-1)T}V_r \gamma g_{(N-1)T}$, and consider two cases.
\begin{itemize}
\item If $N\in J$, so that $|J'| = |J|-1$ and $N-1- |J'| = N- |J|$, write
$$\begin{aligned}
 \frac{1}{T} \int_{(N-1)T}^{NT} 1_{O^c}(g_tqx)\,dt& = \frac{1}{T} \int_{(N-1)T}^{NT}  1_{O^c}\big(g_t({g_{ - (N-1)T}}{h} \gamma{g_{(N-1)T}} )x \big)\,dt\\ &=\frac{1}{T} \int_0^{T} 1_{O^c}\big(g_th (\gamma g_{(N-1)T} x )\big)\,dt.
\end{aligned}
$$
Consequently, 
\eq{induction2}{
\begin{aligned}
&  \left \{q \in g_{-(N-1)T}V_r \gamma g_{(N-1)T}:\frac{1}{T} \int_{(N-1)T}^{NT} 1_{O^c}(g_tqx)\,dt \ge \vre\right \} \\
 & = 
  g_{(N-1)T} \left \{ h \in V_r: \frac{1}{T} \int_0^T 1_{O^c}\big(g_th (\gamma g_{(N-1)T} x )\big)\,dt \ge \vre  \right \} \gamma g_{(N-1)T} \\
  & =  g_{(N-1)T} A_{\gamma g_{(N-1)T} x,\vre} (T,r,O^c) \gamma g_{(N-1)T}.
  \end{aligned}}
Hence, 
by 
applying Corollary \ref{sub count} with
$\gamma g_{(N-1)T} x$ in place of $x$, we can cover the set 
in the left hand side of  \equ{induction2} with at most $e^{\eta T}  {\left(\frac{C(T,O)}\vre\right)^{|J|}}$ Bowen $(NT,r)$-boxes in $P$. Therefore the number of  Bowen $(NT,r)$-boxes needed to cover $A_{x, \vre}(T, O^c, 
J)$ is at most $e^{\eta T}  {\left(\frac{C(T,O)}\vre\right)^{|J|}}$ times the quantity in \equ{uppcov}, which is precisely \equ{uppcovprime}.
\item If $N\notin J$, so that $|J'| = |J|$ and $N-1- |J'| = N- 1-|J|$, the argument is even clearer. By \text{Lemma } \ref{covering}, $g_{-(N-1)T}V_r \gamma g_{(N-1)T}$ can be covered by at most $e^{\eta T}  \left(1 + c_0e^{-{{\lambda_{\min}}} t}\right)$ Bowen $(NT,r)$-boxes.
Hence the number of  Bowen $(NT,r)$-boxes needed to cover $A_{x, \vre}(T, O^c, 
J)$ is at most $e^{\eta T}  \left(1 + c_0e^{-{{\lambda_{\min}}} t}\right)$ times the quantity in \equ{uppcov}, which again is precisely \equ{uppcovprime}.   
\end{itemize}} \end{proof}

Recall that  our goal in this section is to find a covering result for the sets of the form $A_{x, \delta}(NT,O^c)$. 
The following lemma 
reduces this task to a covering result for the sets of the form $A_{x, \vre}(T, O^c,J)$ for $0<\vre < \delta$ and $J\subset \{1,\dots,N\}$. 
   
\begin{lem}\label{main lemma} 
For any  $S \subset X$, $N \in \N$, 
$T>0$, $x \in X$, $0<\delta<1$, and for any $0< \vre < \delta$
$$A_{x, \delta}(NT,S) \subset \bigcup_{J\subset \{1,\dots,N\}: |J| \ge \lceil \left(1 -\frac{1-\delta}{1-\vre} \right) N \rceil } A_{x, \vre}(T, S,J)$$
.
\end{lem}
\begin{proof}
Let $N \in \N$, $r>0$, $T>0$, $x \in X$ and  $0< \vre < \delta<1$.
Also let $h \in  A_{x, \delta}(NT,S) $, and define
$$E: = \big\{j \in \{1, \dots, N\}: h \notin A_{x, \vre}(T,S,\{j\}) \big\}       .$$
 Then 
\begin{align*}
 \delta & {\le}        \frac{1}{NT} \int_{0}^{NT} 1_{S}(g_thx)\,dt  \\
& =\frac{1}{N} \sum_{j \in E}\frac{1}{T} \int_{(j-1)T}^{jT} 1_{S}(g_thx)\,dt +\frac{1}{N} \sum_{j\in\{1,\dots,N \}\smallsetminus E} \frac{1}{T}\int_{(j-1)T}^{jT} 1_{S}(g_thx)\,dt \\
& \le \frac{1}{N} \cdot |E| \cdot \vre +\frac{1}{N} \cdot | \{1,\dots,N \}\smallsetminus E|  
 =\frac{1}{N} \left(|E| \cdot \vre +N-|E|\right).
\end{align*}
This implies
$$|E| \le \frac{1-\delta}{1- \vre} N.$$
Note that it follows immediately from the definition of  $E$ that   $h$ is an element of $A_{x, \vre} (T,S,\{1,\dots, N\}\smallsetminus E)$. Hence, in view of the above inequality we conclude that
$$h \in \bigcup_{J\subset \{1,\dots,N\}: |J| \ge \lceil \left(1 -\frac{1-\delta}{1-\vre} \right) N \rceil } A_{x, \vre}(T,S,J),$$ 
finishing the proof of the lemma. 
\end{proof}

\ignore{\begin{proof} Let $B = g_{-t}V_r \gamma g_t$ be a Bowen $(t,r)$-box. 
In view of the Besicovitch covering property of $P$, any covering of $B$ by balls in $P$ of radius $re^{- \lambda_{\max}t}$ has a subcovering of  index uniformly bounded from above by a fixed constant   (the Besicovitch constant of $P$). The union of those balls is contained in the $re^{-\lambda_{\max} t}$-neighborhood of $B$. But since $B$ is a translate of the exponential image of a box in Lie algebra of $P$ whose  smallest side-length is $re^{-\lambda_{\max} t}$, it follows that the measure of  the $re^{-\lambda_{\max} t}$-neighborhood of $B$ is bounded by a uniform constant times $\nu(B)$, and the lemma follows.\end{proof}}
From the above lemma  combined with Proposition \ref{delta} 
we get the following crucial covering result: 
 \begin{cor}\label{delta 2}
 Let $C_1$ be as in Lemma \ref{coveringballs}, $0< \delta <1$, $0<r\le r_2$, {and let $T_r$ be as in \equ{Tr}.}
Then for any $x\in X$, any {$N \in \N$, any {$T>T_r$},
and for any $0< \vre < \delta$  the set $A_{x,\delta}(NT,O^c)$ can be covered with at most \eq{bowen-ball}
{{\begin{aligned}
& C_1 {e^{{{\lambda_{\max}}} LNT}} {N \choose{\lceil \left(1 -\frac{1-\delta}{1-\vre} \right) N} \rceil} \cdot {
{\left(\frac{C(T,O)}\vre\right)^{\lceil \left(1 -\frac{1-\delta}{1-\vre} \right) N\rceil }}
\cdot \left(1 + c_0 e^{-{{\lambda_{\min}}} T}\right)^{{N - \lceil \left(1 -\frac{1-\delta}{1-\vre} \right) N}\rceil  } } \end{aligned}}}
    balls in $P$ of radius $re^{- {{\lambda_{\max}}} NT}$.
}
\end{cor}
{
\begin{proof}
Let $x\in X$, $N \in \N$, {$T>T_r$},
and let $0< \vre < \delta$.
By the above lemma we have:
$$A_{x, \delta}(NT,O^c) \subset \bigcup_{J\subset \{1,\dots,N\}: |J| \ge \lceil \left(1 -\frac{1-\delta}{1-\vre} \right) N \rceil } A_{x, \vre}(T, O^c,J).$$
Now note that for any $J\subset \{1,\dots,N\}$, if we take any subset $J'$ of $J$, then it follows immediately that 
$A_{x, \vre}(T, O^c,J) \subset A_{x, \vre}(T, O^c,J')$. Therefore, the above inclusion yields the following inclusion:
$$A_{x, \delta}(NT,O^c) \subset \bigcup_{J\subset \{1,\dots,N\}: |J| = \lceil \left(1 -\frac{1-\delta}{1-\vre} \right) N \rceil } A_{x, \vre}(T, O^c,J).$$
Also, by Lemma 
\ref{coveringballs}, every Bowen $(NT,r)$-{boxes} in $P$ can be covered with at most $C_1 e^{({{\lambda_{\max}}}L -\eta)NT}$ balls in $P$ of radius $re^{- {{\lambda_{\max}}} NT}$. From this, combined with Proposition \ref{delta} and the above inclusion we can conclude the proof.
\end{proof}}

 \section{Proof of Theorem \ref{dimension drop 3} }\label{proof}
\begin{proof}[Proof of Theorem \ref{dimension drop 3}]
Let $O$ be an open subset of $X$, $x \in X$, and let $\delta>0$. In view of countable stability of Hausdorff dimension, it suffices to show that for any $0<r\le r_2$, we have
$$\codim S_{x, \delta }(O) \gg \frac{\mu ({\sigma _{r}}O) \cdot {\phi \left(\mu ({\sigma _{r}}O), \sqrt{1-\delta}\right)}}{ \log \frac{1}{r}}$$
where $S_{x, \delta }(O)$ is as in \equ{S} and $\phi$ is as in \equ{defphi}. 
In order to prove the above statement, it is evident that it suffices to demonstrate that for any $0<r\le r_2/5 $ we have
\eq{finalestimate}{\codim S_{x, \delta }(O) \gg \frac{\mu ({\sigma _{5r}}O) \cdot {\phi \left(\mu ({\sigma _{5r}}O), \sqrt{1-\delta}\right)}}{ \log \frac{1}{5r}}.}

If $\mu ({\sigma _{5r}}O)=0$, then the above statement follows immediately. So in this proof we assume always that $\mu ({\sigma _{5r}}O)>0$. 
We start with the following combinatorial lemma:
\begin{lem}\label{n-choose-k}
Let $m={m}(n)\le n$ be a function of $n$  such that $\lim_{n \rightarrow \infty} \frac{m}{n}=z$ for some fixed constant $0<z<1$. Then
{$$ {n \choose m}=o(1) {B\big(\tfrac mn\big)}^n,$$
where \eq{bx}{B(z):=
 \left(\frac{1}{z} \right)^{z}  \left(\frac{1}{1-z} \right)^{1-z} .}}
\end{lem} 
\begin{proof}
Note that $\lim_{n \rightarrow \infty} \frac{m}{n}=z<1$  implies that both $m$ and $n-m$ tend to infinity as $n$ goes to infinity; moreover, $\lim_{n \rightarrow \infty} \frac{n-m}{n}=1-z$. Hence, by using Stirling's approximation we have:
\begin{align*}
{n \choose m} = \frac{n!}{m! (n-m)!}& =(1+o(1)) \frac{\sqrt{2 \pi n} (\frac{n}{e})^n}{\sqrt{2 \pi m} (\frac{m}{e})^m + \sqrt{2 \pi (n-m)} (\frac{n-m}{e})^{n-m}} \\
& =(1+o(1)) \sqrt{\frac{n}{2 \pi m (n-m)}} \left(\frac{n}{m}\right)^m \left(\frac{n}{n-m}\right)^{n-m} \\
&= o(1) \left(\frac{n}{m}\right)^m \left(\frac{n}{n-m}\right)^{n-m}\\
& = o(1) 
{B\big(\tfrac mn\big)}^n,   
\end{align*}
where the third equality above follows from the fact that
$$\lim_{n \rightarrow \infty} \frac{n}{2 \pi m(n-m)}=\lim_{n \rightarrow \infty}  \frac{1}{2 \pi n} \cdot \lim_{n \rightarrow \infty} \frac{n}{m} \cdot \lim_{n \rightarrow \infty}  \frac{n}{n-m} =0 \cdot \frac{1}{z} \cdot \frac{1}{1-z}=0.$$

\end{proof}
{Now {given $0 < \vre< \delta$ set \eq{ze}{z:=1 -\frac{1-\delta}{1-\vre}\quad\Longleftrightarrow \quad \vre=1 -\frac{1-\delta}{1-z}  = \frac{\delta-z}{1-z} .} Lemma \ref{n-choose-k}, applied with $n$ replaced with $N$ and  $m$ replaced with $\left\lceil zN \right\rceil$,  then implies that there exists $N_0 = N_0(z) \in\N$ such that \eq{binomial}{{N \choose \left\lceil zN \right\rceil} \le  B\left(\frac{ \lceil zN \rceil}N\right)^N\quad \text{ for all }N\ge N_0.}}
Take  $0<r \le r_2/5$ and {$T_r$ as in \equ{Tr},} and let {$T>T_r$}. By combining Corollary~\ref{delta 2} with 
\equ{binomial} we get that
for any 
$N\ge N_0$ and any $0< \vre< \delta$, 
\eq{fincov}
{\begin{aligned}A_{x,\delta}(NT,O^c)\text{ can be covered with at most }\quad\\
C_1  e^{LN {{\lambda_{\max}}} T}  \cdot  \beta_{N} ^N\text{ balls in $P$ of radius }re^{- {{\lambda_{\max}}} NT},\end{aligned}
}
where
$$ \beta_{N}  :={B\left(\frac{ \lceil zN \rceil}N\right)} \cdot \left(\frac{C(T,O)}\vre\right)^{\frac{\lceil zN \rceil}{N}} \left(1 + c_0 e^{-{{\lambda_{\min}}} T}\right)^{1-\frac{\lceil zN \rceil  }{N} } .   $$
Note that we have
\eq{betaN}{\lim_{N \rightarrow \infty}   \beta_{N} =     B(z) \cdot  \left(\frac{C(T,O)}\vre\right)^{z}   \left(1 + c_0 e^{-{{\lambda_{\min}}} T}\right)^{1-z}   }
In view of \equ{ze}, \equ{fincov}, \equ{betaN} and Proposition \ref{lsup}, by applying Lemma \ref{Auxillary} with $e^{-{{\lambda_{\max}}} T}$ in place of {$\rho$, $L+ \frac{\log \beta_{N}}{{{\lambda_{\max}}} T}$ in place of $\alpha_N$ 
and $r$ in place of $C$,} we conclude that for any $0< z < \delta$ the Hausdorff dimension of the set $S_{x,\delta}(O)$ is bounded from above by 
$$\begin{aligned}
&L+ \frac{1}{{{{\lambda_{\max}}} T}}   {\log \left({B(z)}    \biggl(\frac{C(T,O)(1-z)}{\delta-z}\biggr)^{z}    \left(1 + c_0 e^{-{{\lambda_{\min}}} T}\right)^{1-z}\right)}\\
& \underset{ \equ{bx}}= L+ \frac{1}{{{{\lambda_{\max}}} T}}  {\log \left(   \biggl(\frac{C(T,O)(1-z)}{z(\delta-z)}\biggr)^{z}   \left(\frac{1 + c_0 e^{-{{\lambda_{\min}}} T}}{1-z}\right)^{1-z}\right).}
\end{aligned}$$

This shows that our objective should be to find $z\in (0,\delta)$ and $T > T_r$ such that the value of 
$$\frac1T\left(z \log   \Big(\frac{z(\delta-z)}{C(T,O)(1-z)}\Big)+ (1-z) \log  \Big(\frac{1-z}{1 + c_0 e^{-{{\lambda_{\min}}} T}}\Big)\right)$$
is the largest possible. We are going to approximate the maximum by first choosing $T$ in a convenient way. Take $T_0 \ge 1$ sufficiently large so that for any $T\ge T_0$ one has
\eq{T0}{c_0e^{-\lambda_{\min} T} < \frac{1}{T}}
{(note that $T_0$ depends only on $G$, $F_+$ and $P$)}, and set \eq{Tfinal}{ T:=\max \left(\frac{8T_r}{\mu ({\sigma _{5r}}O)}, T_0\right). }
Then
{
$${ \frac{1}{T}< \frac{1+T_r}{T} \le \frac{2T_r}{T} \le \frac{\mu ({\sigma _{5r}}O)}{4},        }$$}
which, in combination with  \equ{T0}, yields
$${ 1 + c_0 e^{-{{\lambda_{\min}}} T} \le 1+ \frac{1}{T} < 1+ \frac{\mu ({\sigma _{5r}}O)}{4}
}$$
and
$$
{
\begin{aligned}
C(T,O)& = 1-\mu ({\sigma _{5r}}O) +\frac{T_r+1}{T} +c_0 e^{-{{\lambda_{\min}}} T} \\
& \le  1-\mu ({\sigma _{5r}}O) +\frac{\mu ({\sigma _{5r}}O)}{4} +\frac{\mu ({\sigma _{5r}}O)}{4}  =   1-\frac{\mu ({\sigma _{5r}}O)}{2}.
\end{aligned}
}$$
Therefore  for $T$ as in \equ{Tfinal} the codimension of $S_{x,\delta}(O)$ is  for any $0< z < \delta$ bounded from below by 
$$
\frac{1}{{{{\lambda_{\max}}} T}}  \left(z \log   \Biggl(\frac{z(\delta-z)}{\left(1-\frac{\mu ({\sigma _{5r}}O)}{2}\right)(1-z)}\Biggr)+ (1-z) \log  \left(\frac{1-z}{1+ \frac{\mu ({\sigma _{5r}}O)}{4}}\right)\right).$$
Note that the second summand in the above expression is always negative; thus it makes sense to try choosing $0<z <\delta$ in a way that ensures that
the first summand is maximized and is positive if possible (this condition could prove to be unsuccessful for any $0< z <\delta$, contingent upon $\mu ({\sigma _{5r}}O)$ and $\delta$; in this case we will not achieve a dimension drop).

\ignore{By taking the logarithm of the above inequality, we obtain:
\eq{logsecond}
{\begin{aligned}
 \log C'(T,O) & \le \log \left(  1-\frac{\mu ({\sigma _{5r}}O)}{2}\right) \\
\\
& \le  -\frac{\mu ({\sigma _{5r}}O)}{2}
\end{aligned}}\bigskip
estimate the minimum of $$\left(\frac{C'(T,O)}{ c \cdot \vre} \right)^c  \cdot \left(\frac{1 + c_0 e^{-{{\lambda_{\min}}} T}}{1-c}\right)^{1-c}=\left(\frac{C'(T,O)}{\left(1- \frac{1-\delta}{1- \vre}\right) \cdot \vre} \right)^{\left(1- \frac{1-\delta}{1- \vre}\right)}  \cdot \left(\frac{1 + c_0 e^{-{{\lambda_{\min}}} T}}{\frac{1-\delta}{1- \vre}}\right)^{\frac{1-\delta}{1- \vre}}$$ over the interval $0<\vre< \delta$ in terms of $\delta$. For simplicity, denote this minimum by $m_\delta$.
Note that since $c<1$, we have $ \left(\frac{1 + c_0 e^{-{{\lambda_{\min}}} T}}{1-c}\right)^{1-c}>1$.
Hence, in order to obtain dimension drop from \equ{dim1}, we should choose $0<\vre <\delta$ in a way that ensures that
$\left(\frac{C'(T,O)}{ c \cdot \vre} \right)^c$ is minimized and is less than 1 if possible (this condition could prove to be unsuccessful for any $0< \vre <\delta$, contingent upon $C'(T, O)$ and $\delta$; in this case we won't achieve a dimension drop with the chosen $\delta$).}

We will solve the latter problem approximately by  finding $0<z< \delta$ which maximizes the ratio $\frac{z(\delta-z)}{1-z}$. An elementary calculus exercise shows that for that one should take   $z =1-\sqrt{1-\delta}$, so that  $\frac{z(\delta-z)}{1-z} = (1-\sqrt{1-\delta})^2$. Denoting $s = \sqrt{1-\delta}$ and $y = \mu ({\sigma _{5r}}O)$, we get an estimate
$$
\begin{aligned}
\codim S_{x,\delta}(O) \ge\  \ &\frac{1}{{{{\lambda_{\max}}} T}}  \left((1-s) \log   \biggl(\frac{(1-s)^2}{1-\frac{y}{2}}\biggr)+ s \log  \left(\frac{s}{1+ \frac{y}{4}}\right)\right)\\
\underset{\equ{defphi}}= \ &\frac{1}{{{{\lambda_{\max}}} T}} \phi(y,s).
\end{aligned}.
$$
It remains to observe that $$T\ \underset{\equ{Tfinal}}\ll \ \frac{8T_r}{\mu ({\sigma _{5r}}O)} \ \underset{\equ{Tr}}\ll \ \frac{\log \frac{1}{r}}{\mu ({\sigma _{5r}}O)}\  \underset{r\le \frac {r_2} {5}<1/40}{\ll} \ \frac{\log \frac{1}{5r}}{\mu ({\sigma _{5r}}O)};$$
thus \equ{finalestimate} follows, which ends the proof of Theorem  \ref{dimension drop 3}.}\end{proof}

\bibliographystyle{alpha}

\end{document}